\documentclass[12pt]{amsart}
\usepackage{amssymb,amsmath,mathrsfs,enumerate,mathtools}
\usepackage{tikz-cd}
\usepackage[pdftex,bookmarks=true]{hyperref}

\theoremstyle{plain}

\DeclareMathOperator*{\esssup}{ess\,sup}
\newtheorem{theorem}{Theorem}[section]
\newtheorem{problem}[theorem]{Problem}
\newtheorem{proposition}[theorem]{Proposition}
\newtheorem{lemma}[theorem]{Lemma}
\newtheorem{corollary}[theorem]{Corollary}
\theoremstyle{definition}

\newtheorem{example}[theorem]{Example}
\newtheorem{definition}[theorem]{Definition}

\numberwithin{equation}{section}

\DeclareRobustCommand{\rchi}{{\mathpalette\irchi\relax}}
\newcommand{\irchi}[2]{\raisebox{\depth}{$#1\chi$}}

\author[M.~Tantrawan]{Made Tantrawan}
\address{\textit{a} Department of Mathematics,   National University of Singapore, Singapore 119076 \newline
\textit{b} Department of Mathematics, Faculty of Mathematics and Natural Sciences, Universitas Gadjah Mada, Indonesia 55281}
\email{made.tantrawan@ugm.ac.id}

\author[D.~Leung]{Denny H.~Leung}
\address{Department of Mathematics, National University of Singapore, Singapore 119076}
\email{matlhh@nus.edu.sg}

\title{On closedness of convex sets in Banach lattices}

\thanks{The first author is supported by NUS Research Scholarship. The second author is partially supported by AcRF grant R-146-000-242-114.}


\date{\today}

\begin{document}

\begin{abstract}
Let $X$ be a Banach lattice. A well-known problem arising from the theory of risk measures asks when order closedness of a convex set in $X$ implies closedness with respect to the topology $\sigma(X,X_n^\sim)$, where $X_n^\sim$ is the order continuous dual of $X$. Motivated by the solution in the Orlicz space case, we introduce two relevant properties: the disjoint order continuity property ($DOCP$) and the order subsequence splitting property ($OSSP$). We show that when $X$ is monotonically complete with $OSSP$ and $X_n^\sim$ contains a strictly positive element, every order closed convex set in $X$ is $\sigma(X,X_n^\sim)$-closed  if and only if $X$ has $DOCP$ and either $X$ or $X_n^\sim$ is order continuous. This in turn occurs if and only if either $X$ or the norm dual $X^*$ of $X$ is order continuous. We also give a modular condition under which a Banach lattice has $OSSP$. In addition, we also give a characterization of $X$ for which order closedness of a convex set in $X$ is equivalent to closedness with respect to the topology $\sigma(X,X_{uo}^\sim)$, where $X_{uo}^\sim$ is the unbounded order continuous dual of $X$.
\end{abstract}

\maketitle

\section{Introduction}

\subsection{Background and motivations}
This paper is motivated by recent developments in the theory of risk measures. One of important problems in the theory of risk measures asks when a coherent risk measure admits a Fenchel-Moreau dual representation. Note that any coherent risk measure is a proper convex functional. For a locally convex topological space $(X,\tau)$, the Fenchel-Moreau formula \cite[Theorem 1.11]{Br} asserts that a proper convex functional $\rho:X\to(-\infty,\infty]$ admits a dual representation via the topological dual $(X,\tau)^*$ if and only if $\rho$ is $\tau$-lower semicontinuous, i.e., $C_\lambda:=\left\{\rho\leq\lambda\right\}$ is $\tau$-closed for every $\lambda\in\mathbb{R}$.  In a Banach lattice $X$, the order continuous dual $X_n^\sim$ of $X$ is one of the topological duals that has been intensively studied recently, related to this duality problem (see \cite{BF, De, GL, GLMX, GLX, GX, Ow}). When $X=L^\infty$, it was proved in \cite{De} that a coherent risk measure $\rho$ admits a Fenchel-Moreau dual representation  via $X_n^\sim=L^1$ if and only if $\rho$ has the Fatou property, i.e., $\rho(f)\leq\liminf_n\rho(f_n)$ whenever $\{f_n\}$ order converges to $f$ in $X$. Hence, in order to solve the Fenchel-Moreau duality problem for a general Banach lattice, it is natural to ask the following problem which was stated as an open question in \cite[p. 3585]{Ow}.

\begin{problem}\label{mainproblem0}
	Let $\rho$ be a proper convex functional on a Banach lattice $X$. Does the Fatou property of $\rho$ imply  $\sigma(X,X_n^\sim)$ lower semicontinuity of $\rho$?
\end{problem}

Since the Fatou property of $\rho$ is equivalent to order closedness of the sublevel sets $C_\lambda$, Problem \ref{mainproblem0} is closely related to the following problem.

\begin{problem}\label{mainproblem}
	Let $X$ be a Banach lattice. Is it true that every order closed convex set in $X$ is $\sigma(X,X_n^\sim)$-closed?
\end{problem}

Note that an affirmative answer to Problem \ref{mainproblem} will give an affirmative answer to Problem \ref{mainproblem0}. When $X=L^\varphi$ is an Orlicz space over a nonatomic probability space, Delbaen and Owari \cite{DO} obtained a partial positive solution to Problem \ref{mainproblem}. Completing the result for Orlicz spaces $X=L^\varphi$, Gao et al. \cite{GLX} proved that Problem \ref{mainproblem0} and Problem \ref{mainproblem} have affirmative answers if and only if either $\varphi$ or its conjugate $\varphi^*$ satisfies the so-called $\Delta_2$-condition, equivalently, if and only if either $X$ or $X_n^\sim$ has order continuous norm. In this paper, we extend the result into a large class of Banach lattices. 

The paper is organized as follows. In Section 2, we investigate some necessary and sufficient conditions for Banach lattices to have an affirmative answer for Problem \ref{mainproblem}. Motivated by \cite{GLX}, we introduce two properties which we call the order subsequence splitting property ($OSSP$) and the disjoint order continuity property ($DOCP$), respectively. The main result of the paper is that for a monotonically complete Banach lattice $X$ which has $OSSP$ and admits a strictly positive order continuous functional, Problem \ref{mainproblem} has an affirmative answer if and only if $X$ has $DOCP$ and either $X$ or $X_n^\sim$ is order continuous, if and only if $X$ or $X^*$ is order continuous. In Section 3, we apply the results in Section 2 into some known Banach lattices, namely Musielak-Orlicz spaces and Ces\`{a}ro function spaces. We also explain why Theorem \ref{P1-characterization} can be seen as a generalization of the result in \cite{GLX}. In the last section, we give a characterization of when order closedness of a convex set in a Banach lattice $X$ implies closedness with respect to the topology  $\sigma(X,X_{uo}^\sim)$, where $X_{uo}^\sim$ is the unbounded order continuous dual of $X$.

\subsection{Basic definitions and facts}

We refer to \cite{MN, Sc, AB, Wn}  for basic definitions and facts on Banach lattices.  Let $X$ be a Banach lattice. A set $E$ in $X$ is said to be order bounded if there exists $h\in X$ such that $|f|\leq h$ for every $f\in E$. $X$ is called Dedekind complete (respectively, $\sigma$-Dedekind complete) if every non-empty order bounded set (respectively, sequence) has a supremum and an infimum in $X$. A sublattice $E$ of $X$ is said to be order dense in $ X$ if for each $f\in X$ with $f>0$, there is some $h\in E$ such that $0<h\leq f$.

A net $\{f_{\alpha}\}$ in $X$ is said to order converge to $f$ in $X$, written as $f_\alpha\xrightarrow{o}f$, if there exists a net $\{h_\gamma\}$ such that $h_\gamma\downarrow0$ and for every $\gamma>0$, there exists $\alpha_0$ such that $|f_\alpha-f|\leq h_\gamma$ for $\alpha\geq\alpha_0$. When $X$ is $\sigma$-Dedekind complete, a sequence  $\{f_n\}$ order converges to $f$ in $X$ if and only if there exists a decreasing sequence $\{h_n\}$ in $X$ such that $|f_n-f|\leq h_n$ for every $n$. The order continuous dual $ X_n^\sim$ of $ X$ is the collection of all linear functionals $g$ on $ X$ which are order continuous, i.e., $\langle g,f_\alpha\rangle:=g(f_\alpha)\to0$ whenever $f_\alpha\xrightarrow{o}0$. Note that $X_n^\sim$ is a Banach lattice and an ideal in the norm dual $ X^*$ of $ X$.  Denote by $ X_a$ the order continuous part of $X$, that is, the set of all $f\in{X}$ such that $\|f_\alpha\|_{X}\to0$ whenever $f_\alpha\xrightarrow{o}0$ and $|f_\alpha|\leq |f|$. We say that $X$ is order continuous if $ X_a= X$, or equivalently, $ X_n^\sim= X^*$ (\cite[Theorem 2.4.2]{MN}). 

A net $\{f_\alpha\}$ in $X$ is said to uo-converge to $f$ in $X$, we write $f_\alpha\xrightarrow{uo}0$, if $|f_\alpha-f|\wedge h\xrightarrow{o}0$ for all $h\in X_+$. The unbounded order continuous dual (uo-dual) ${X}_{uo}^\sim$ of $X$ is the collection of all linear functionals $g$ on $X$ such that $\langle g,f_\alpha\rangle\to0$ whenever $f_\alpha\xrightarrow{uo}0$ and $\{f_\alpha\}$ is norm bounded. For any Banach lattice $X$, $X_{uo}^\sim$ is the order continuous part of $X_n^\sim$, i.e., $X_{uo}^\sim=\left(X_{n}^\sim\right)_a$ (\cite[Theorem 2.3.]{GLX1}). See \cite{GLX1, GTX, GX1} for more details on the concepts of uo-convergence and uo-dual.

A Banach lattice $X$ is said to be monotonically complete if $\sup_\alpha f_\alpha$ exists for every increasing norm bounded net $\{f_\alpha\}$ in $X$. We say that $X$ has the weak Fatou property if  there exists $r>0$ such that every increasing net $\{f_\alpha\}$ with the supremum $f\in X$ satisfies $\|f\|_X\leq r\sup_\alpha \|f_\alpha\|_X$. Note that a monotonically complete Banach lattice $X$ is Dedekind complete and has the weak Fatou property (\cite[Theorem 2.4.19]{MN}). When $X$ is Dedekind complete and $X_n^\sim$ separates points of $X$ (i.e., for every $f\in X$ with $f\neq 0$, there exists $g\in X_n^\sim$ such that $\langle g,f\rangle\neq 0$), $X$ is monotonically complete if and only if $X=(X_n^\sim)_n^\sim$ with equivalence of norms (\cite[Theorem 2.4.22]{MN}).

\section{Order closedness and  $\sigma( X, X_n^\sim)$-closedness of convex sets in Banach lattices}

Let $X$ be a Banach lattice. For any $E\subseteq X$, we define its order closure $\overline{E}^o$ to be the set of all $f\in X$ such that there exists a net $\{f_\alpha\}$ in $E$ which order converges to $f$. We say that $E\subseteq  X$ is order closed if $\overline{E}^o=E$. A net $\{f_\alpha\}$ in $ X$ is said to $|\sigma|( X, X_n^\sim)$-converge to $f$ in $ X$, written as $f_\alpha\xrightarrow{|\sigma|( X, X_n^\sim)}f$, if $\{|f_\alpha-f|\}$ $\sigma( X, X_n^\sim)$-converges to $0$. Observe that for any $E\subseteq  X$, $$\overline{E}^o\subseteq \overline{E}^{|\sigma|_s( X, X_n^\sim)}\subseteq\overline{E}^{\sigma_s( X, X_n^\sim)}\subseteq  \overline{E}^{\sigma( X, X_n^\sim)}=\overline{E}^{|\sigma|( X, X_n^\sim)}$$ where $\overline{E}^{\sigma_s( X, X_n^\sim)}$ and $\overline{E}^{|\sigma|_s( X, X_n^\sim)}$ are the $\sigma( X, X_n^\sim)$-sequential closure and the \linebreak $|\sigma|( X, X_n^\sim)$-sequential closure of $E$, respectively. The last equality comes from Mazur's theorem and the fact that the topological dual of $ X$ under $|\sigma|( X, X_n^\sim)$ is $ X_n^\sim$ (\cite[Theorem 3.50]{AB}). It follows that every $\sigma( X, X_n^\sim)$-closed convex set in $ X$ is order closed. Therefore, if Problem \ref{mainproblem} has an affirmative answer, the order closedness of a convex set $C$ in $ X$ is equivalent to the $\sigma( X, X_n^\sim)$-closedness of $C$.  

\begin{definition}
Let $ X$ be a Banach lattice. We say that order closedness of convex sets in $ X$ is $\sigma( X, X_n^\sim)$ determined if every order closed convex set $C$ is $\sigma( X, X_n^\sim)$-closed. We call this property $P1$ in brief.
\end{definition}

With this terminology, we are interested precisely in the problem of identifying the Banach lattices $X$ with property $P1$. We begin with a necessary condition for $X$ to have property $P1$. For any $x\in X$, denote by $B(x)$ the band generated by $x$. When $X$ is $\sigma$-Dedekind complete, each $B(x)$ is a projection band (see, e.g., \cite[Corollary 2, p. 64]{Sc}). 

\begin{lemma}\label{split-seq}
	Suppose that $X$ is Dedekind complete with the weak Fatou property and $X_n^\sim$ separates points of $X$. If both $ X$ and $ X_n^\sim$ are not order continuous, then there exist a norm bounded set of disjoint positive elements $A=\{x_n\}_{n\geq 1}\cup \{w_0\}\cup\{w_{ij}\}_{i,j\geq1}$ in $ X$ and a norm bounded set of disjoint positive elements $B=\{y_n\}_{n\geq 1}\cup \{z_0\}\cup\{z_{ij}\}_{i,j\geq1}$ in $ X_n^\sim$ such that
	\begin{enumerate}[$(a)$]
		\item the order sums $(o)\sum_{n\geq 1}x_n$ and $(o)\sum_{i,j\geq1}z_{ij}$ belong to $ X$ and $ X_n^\sim$ respectively,
		\item  $\langle y_n,x_n\rangle =\langle z_0,w_0\rangle=\langle z_{ij},w_{ij}\rangle=1$ for all $n,i,j\geq1$, and $\langle y,x \rangle=0$ for the remaining pairs $(x,y)\in A\times B$. 
	\end{enumerate}
\end{lemma}
\begin{proof}
Since $X$ is not order continuous, we may apply \cite[Theorem 5.14, p. 94]{Sc} to obtain a normalized disjoint sequence $\{f_n\}$ in $X_+$ such that the order sum $(o)\sum_nf_n$ belongs to $X$. Let $\bar{f}_1:=(o)\sum_nf_{2n-1}$ and $\bar{f}_2:=(o)\sum_nf_{2n}$.  Denote by $P_1$ and $P_2$ the band projections onto $B(\bar{f}_1)$ and $B(\bar{f}_2)$, respectively. Since $\bar{f}_1$ and $\bar{f}_2$ are disjoint, $P_3:=I_X-P_1-P_2$ is also a band projection, where $I_X$ is the identity operator on $X$. Observe that 
\[
X=P_1(X)\oplus P_2(X)\oplus P_3(X).
\]
Since $\{f_{2n-1}\}_n$ and $\{f_{2n}\}_n$ are normalized disjoint sequences in $P_1(X)=B(\bar{f}_1)$ and $P_2(X)=B(\bar{f}_2)$ respectively, both $P_1(X)$ and $P_2(X)$ are not order continuous by \cite[Theorem 5.14, p. 94]{Sc}. For $i=1,2,3$, let $P_i^*$ be the dual operator of $P_i$. Then each $P_i^*$ is a band projection and
\[
X_n^\sim=P_1^*\left(X_n^\sim\right)\oplus P_2^*\left(X_n^\sim\right)\oplus  P_3^*\left(X_n^\sim\right).
\]
Since $X_n^\sim$ is not order continuous, there exists $s\in\{1,2,3\}$ such that $P_s^*(X_n^\sim)$ is not order continuous. Pick $t\in \{1,2\}\backslash\{s\}$. Then $P_t(X)$ is not order continuous and $\langle y,x\rangle=0$ for every $(x,y)\in \left(P_s^*(X_n^\sim)\times P_t(X)\right)\cup \left(P_t^*(X_n^\sim)\times P_s(X)\right)$.

Since $P_t(X)$ is not order continuous, we may apply \cite[Theorem 5.14, p. 94]{Sc} to obtain a normalized disjoint positive sequence $\{x_n\}_{n\geq 1}$ in $P_t(X)$ such that the order sum $(o)\sum_{n\geq 1}x_n$ belongs to $X$. Since $X$ is Dedekind complete with the weak Fatou property and $X_n^\sim$ separates points of $X$, by \cite[Lemma 2.4.20]{MN} there exists $r\geq 1$ such that
\[
\|x\|_X\leq r^2\sup\{\langle x',|x|\rangle:x'\in (X_n^\sim)_+,\ \|x'\|_{X_n^\sim}\leq 1\}
\]
for all $x\in X$. Then we can choose a norm bounded sequence of positive elements $\{y'_n\}_{n\geq 1}\subseteq  X_n^\sim$ such that $\langle y'_n,x_n\rangle=1$ for all $n$. For every $n$, let $Q_n$ be the band projection onto $B(x_n)\subseteq P_t(X)$. Set $y_n=Q_n^*(y_n')$ where $Q_n^*$ is the dual operator of $Q_n$. Then $\{y_n\}_{n\geq 1}$ is a norm bounded disjoint positive sequence in $X_n^\sim$ such that $y_n\wedge y=0$ for all $y\in P^*_s(X_n^\sim)_+$, $\langle y_n,x_n\rangle=1$ and $\langle y_n,x_m\rangle=0$ for all $n\neq m$. 

Since $P^*_s(X_n^\sim)$ is not order continuous, we may apply \cite[Theorem 5.14, p. 94]{Sc} again to obtain a normalized disjoint positive sequence $\{z'_0\}\cup\{z'_{ij}\}_{i,j\geq1}$ in $P^*_s(X_n^\sim)\subseteq X_n^\sim$ such that the order sum $(o)\sum_{i,j\geq1}z'_{ij}$ belongs to $X_n^\sim$. Then there is a norm bounded sequence $\{w'_0\}\cup\{w'_{ij}\}_{i,j\geq1}$ in $P_s(X)$ such that $\langle z'_0,w'_0\rangle=\langle z'_{ij},w'_{ij}\rangle= 2$ for every $i,j\geq 1$. Applying \cite[Proposition 2.3.1]{MN} to sequences $\{z'_0\}\cup\{z'_{ij}\}_{i,j\geq1}$ and $\{w'_0\}\cup\{w'_{ij}\}_{i,j\geq1}$, we can find a norm bounded disjoint positive sequence $\{w_0\}\cup\{w_{ij}\}_{i,j\geq1}$ in $P_s(X)$ and a subsequence $\{z_0\}\cup\{z_{ij}\}_{i,j\geq1}$ of $\{z'_0\}\cup\{z'_{ij}\}_{i,j\geq1}$ such that $\langle z_0,w_0\rangle=\langle z_{ij},w_{ij}\rangle= 1$ and $\langle z_0,w_{ij}\rangle=\langle z_{ij},w_{0}\rangle=\langle z_{ij},w_{nm}\rangle=0$ for all $(i,j)\neq (n,m)$. Set $A=\{x_n\}_{n\geq 1}\cup \{w_0\}\cup\{w_{ij}\}_{i,j\geq1}$ and $B=\{y_n\}_{n\geq 1}\cup \{z_0\}\cup\{z_{ij}\}_{i,j\geq1}$. Then they are norm bounded sets of disjoint positive elements  in $X$ and $ X_n^\sim$, respectively, which satisfy $(a)$ and $(b)$.
\end{proof}

For the next lemma, we assume that $X$ is Dedekind complete with the weak Fatou property, $X_n^\sim$ separates points of $X$, and both $X$ and $ X_n^\sim$ are not order continuous. Let $\{x_n\}_{n\geq 1}\cup \{w_0\}\cup\{w_{ij}\}_{i,j\geq1}$ and $\{y_n\}_{n\geq 1}\cup \{z_0\}\cup\{z_{ij}\}_{i,j\geq1}$ be norm bounded sequences obtained in Lemma \ref{split-seq}. Observe that for any $x\in X$, 
\[
\sum_{i,j\geq 1}\left|\langle z_{ij},x\rangle\right|\leq \langle \bar{z},|x|\rangle\leq\|x\|_{ X}\|\bar{z}\|_{ X_n^\sim},
\]
where $\bar{z}:=(o)\sum_{i,j\geq1}z_{ij}$. Then the map $T: X\to \ell^\infty \oplus \mathbb{R}\oplus \ell^1(\mathbb{N}\times\mathbb{N})$ defined by
\[
T(x)=\left\{\langle y_n,x\rangle\right\}_{n\geq 1}\oplus \langle z_0,x\rangle\oplus \left\{\langle z_{ij},x\rangle\right\}_{i,j\geq 1}
\]
is a positive bounded linear operator on $ X$. 

For any $x\in X$, we write $x\sim(\lambda,b)$ if there are $\lambda\in\mathbb{R}$ and $b=\{b(i,j)\}_{i,j\geq1}\in\ell^1(\mathbb{N}\times\mathbb{N})$ such that 
\[
\lambda\geq0,\ b\geq0,\ \sum_{i}2^i\|b_i\|_1=1,
\]
\[
a\geq-\lambda,\ v\geq\lambda b\ \text{and}\ u\geq\lambda\sum_{i=1}^l 4^iS(b_i)\ \text{for all}\ l\geq1, 
\]
where $T(x)=u\oplus a\oplus v$, $b_i=\{b(i,j)\}_j$ for every $i$, and $S(\{a_{j}\}_j)=\left\{\sum_{j=1}^na_j\right\}_{n}$ for every $\{a_j\}_j$.

\begin{lemma}\label{orderclosednotsigmaclosed}
Let $C=\{x\in X:x\sim(\lambda,b)\ \text{for some}\ \lambda\in\mathbb{R}\ \text{and}\ b\in\ell^1(\mathbb{N}\times\mathbb{N})\}.$ Then $C$ is convex, not $\sigma(X,X_n^\sim)$ closed, and $f\in C$ whenever there is a norm bounded net $\{f_\alpha\}$ in $C$ such that ${|f_\alpha-f|} \xrightarrow{\sigma(X,X_n^\sim)}0$.
\end{lemma}
\begin{proof}
  Using the same steps as in the proof of \cite[Lemma 3.6]{GLX}, one can show that $C$ is convex and $-w_0\in \overline{C}^{\sigma( X, X_n^\sim)}\backslash C$. In particular, $C$ is not $\sigma( X, X_n^\sim)$-closed.	

Now, let $f\in X$ and $\{f_\alpha\}_{\alpha\in I}$ be a norm bounded net in $C$ such that ${|f_\alpha-f|} \xrightarrow{\sigma(X,X_n^\sim)}0$. Write
	\[
	T(f_\alpha)=u_\alpha\oplus a_\alpha\oplus v_\alpha\quad\text{and}\quad T(f)=u\oplus a\oplus v.
	\]
	For each $n\geq1$, denote by $u(n)$ the $n$-th coordinate of a vector $u\in\ell^\infty$. Note that
	\[
	\lim_\alpha u_\alpha(n)=\lim_\alpha\langle y_n,f_\alpha\rangle=\langle y_n,f\rangle=u(n)
	\]
	for every $n\geq1$. Moreover, since $\{f_\alpha\}_{\alpha\in I}$ is norm bounded in $ X$, $\{u_\alpha\}_{\alpha\in I}$ is norm bounded in $\ell^\infty$. It follows that $\{u_\alpha\}_{\alpha\in I}$ $\sigma(\ell^\infty,\ell^1)$-converges to $u$. Similarly, $\{a_\alpha\}_{\alpha\in I}$ converges to $a$. For any $\{c_{ij}\}_{i,j\geq 1}\in\ell^\infty(\mathbb{N}\times\mathbb{N})$,
	\begin{eqnarray*}
		\left|\langle\{c_{ij}\},v_\alpha-v\rangle\right|=\left|\sum_{i,j\geq 1}c_{ij}\langle z_{ij},f_\alpha-f\rangle\right|&\leq&\sum_{i,j\geq 1}|c_{ij}|\left\langle z_{ij},|f_\alpha-f|\right\rangle\\
		&\leq&\sup_{i,j\geq 1}|c_{ij}|\left\langle\bar{z},|f_\alpha-f|\right\rangle\to0.
	\end{eqnarray*}
	Hence $\{v_\alpha\}_{\alpha\in I}$ converges to $v$ with respect to the topology $\sigma(\ell^1(\mathbb{N}\times\mathbb{N}),\ell^\infty(\mathbb{N}\times\mathbb{N}))$. 
	
	Denote by $\text{co}(E)$ the set of all convex combinations of elements in $E$. Observe that for every $\alpha\in I$, 
	\[
	v\in \overline{\text{co}(\{v_\beta\}_{\beta\geq\alpha})}^{\sigma(\ell^1(\mathbb{N}\times\mathbb{N}),\ell^\infty(\mathbb{N}\times\mathbb{N}))}=\overline{\text{co}(\{v_\beta\}_{\beta\geq\alpha})}^{\|\cdot\|_{\ell^1(\mathbb{N}\times\mathbb{N})}}.
	\]
	Then for every $(\alpha,n)\in I\times \mathbb{N}$, there exists $v_{\alpha,n}\in \text{co}\{v_\beta\}_{\beta\geq\alpha}$ such that 
	\[
	\|v_{\alpha,n}-v\|_{{\ell^1(\mathbb{N}\times\mathbb{N})}}\leq\frac{1}{n}.
	\]
	For every $(\alpha,n)\in I\times \mathbb{N}$, let $f_{\alpha,n}$, $u_{\alpha,n}$ and $a_{\alpha,n}$ be the corresponding convex combinations in $\text{co}(\{f_\beta\}_{\beta\geq\alpha})$, $\text{co}(\{u_\beta\}_{\beta\geq\alpha})$ and $\text{co}(\{a_\beta\}_{\beta\geq\alpha})$, respectively. Then $\{v_{\alpha,n}\}_{(\alpha,n)\in I\times\mathbb{N}}$ norm converges to $v$,  $\{f_{\alpha,n}\}_{(\alpha,n)\in I\times\mathbb{N}}$ $\sigma( X, X_n^\sim)$-converges to $f$, \linebreak $\{u_{\alpha,n}\}_{(\alpha,n)\in I\times\mathbb{N}}$ $\sigma(\ell^\infty,\ell^1)$-converges to $u$ and $\{a_{\alpha,n}\}_{(\alpha,n)\in I\times\mathbb{N}}$ converges to $a$.
	
	Now, since $(\ell^\infty,\sigma(\ell^\infty,\ell^1))$ and $(\ell^1,\|\cdot\|_{\ell^1(\mathbb{N}\times\mathbb{N})})$ are metrizable on norm bounded sets, we can find a sequence $\{\alpha_k,n_k\}\subseteq  I\times\mathbb{N}$ such that
	\[
	u'_k:=u_{\alpha_k,n_k}\xrightarrow{\sigma(\ell^\infty,\ell^1)}u,\quad a'_k:=a_{\alpha_k,n_k}\rightarrow a\quad \text{and}\quad v'_k:=v_{\alpha_k,n_k}\xrightarrow{\|\cdot\|_{\ell^1(\mathbb{N}\times\mathbb{N})}}v.
	\]
	For each $k\in\mathbb{N}$, let $f'_k:=f_{\alpha_k,n_k}\sim (\lambda_k,b_k)$ and write $b_{k i}=\{b_k(i,j)\}_{j}$ for each $i\in\mathbb{N}$. Choose $M$ so that $\|u'_k\|_\infty\leq M$ for all $k\in\mathbb{N}$. If $l\geq 1$, then
	\begin{equation}\label{counter111}
	M\geq u'_k(n)\geq \lambda_k\sum_{i=1}^l 4^iS(b_{k i})(n)\to \lambda_k\sum_{i=1}^l 4^i\|b_{k i}\|_1
	\end{equation}
	as $n\to\infty$. It follows that $M\geq \lambda_k \sum_{i=1}^l 4^i\|b_{k i}\|_1\geq \lambda_k \sum_{i=1}^l 2^i\|b_{k i}\|_1=\lambda_k\geq0$ and hence, $\{\lambda_k\}$ is a bounded sequence. By passing to a subsequence, we may assume that $\{\lambda_k\}$ converges to some $\lambda\geq0$. 
	
	If $\lambda=0$, it is easy to check that $f\sim (\lambda,b)$ for any $b\in \ell^1(\mathbb{N}\times\mathbb{N})$ such that $\sum_{i=1}^l 2^i\|b_i\|_1=1$ where $b_i=\{b(i,j)\}_j$. Hence, $f\in C$. Now, suppose that $\lambda>0$. Since $v'_k\geq\lambda_kb_k\geq0$ for all $k\in\mathbb{N}$ and $\{v'_k\}$ norm converges in $\ell^1(\mathbb{N}\times\mathbb{N})$, it follows that $\{\lambda_kb_k\}$ is relatively norm compact in $\ell^1(\mathbb{N}\times\mathbb{N})$. By passing to a subsequence again, we may assume that $\{\lambda_kb_k\}$ norm converges to some $d\in \ell^1(\mathbb{N}\times\mathbb{N})$. Set $b=\frac{d}{\lambda}$. Then $b\geq0$ and $\{b_k\}$ norm converges to $b$. We claim that $f\sim(\lambda,b)$. Clearly, for any $i\geq1$, $2^i\|b_{ki}\|_1\to2^i\|b_i\|_1$ as $k\to\infty$. Choose $k_0$ such that $\lambda_k\geq\frac{\lambda}{2}$ for all $k\geq k_0$. By (\ref{counter111}), if $k\geq k_0$, then $0\leq 2^i\|b_{ki}\|_1\leq\frac{M}{\lambda 2^{i-1}}$ for any $i\geq1$. From the dominated convergence theorem, we obtain that
		\[
		\sum_i2^i\|b_i\|_1=\lim_{k\to\infty}\sum_i2^i\|b_{ki}\|_1=1.
		\]
		Furthermore,
		\[
		a=\lim_{k\to\infty} a'_k\geq-\lim_{k\to\infty}\lambda_k=-\lambda
	\quad\text{and}\quad 
		v=\lim_{k\to\infty} v'_k\geq\lim_{k\to\infty}\lambda_kb_k=\lambda b.
		\]
		Note that, for each $n$ and $i$, $S(b_{ki})(n)\to S(b_i)(n)$ as $k\to\infty$. Then for any $l\geq 1$,
		\[
		u(n)=\lim_{k\to\infty} u'_k(n)\geq\lim_{k\to\infty}\lambda_k\sum_{i=1}^l4^iS(b_{ki})(n)=\lambda\sum_{i=1}^l4^iS(b_{i})(n).
		\]
		It follows that
		\[
		u\geq \lambda\sum_{i=1}^l4^iS(b_{i})\quad\text{for any}\ l\geq 1.
		\]
		Thus, $f\sim (\lambda,b)$ and hence, $f\in C$, as desired.
\end{proof}

As a consequence, we have the following necessary condition for property $P1$.

\begin{theorem}\label{sufficient-P1}
	Suppose that $ X$ is Dedekind complete with the weak Fatou property and $X_n^\sim$ separates points of $X$. If $ X$ has property $P1$, then either $ X$ or $ X_n^\sim$ is order continuous.
\end{theorem}
\begin{proof}
Suppose that both $X$ and $X_n^\sim$ are not order continuous. Let $C$ be the set defined in Lemma \ref{orderclosednotsigmaclosed}. Then $C$ is convex and not $\sigma(X,X_n^\sim)$-closed. Let $f$ be an element in the order closure of $C$. There exists a net $\{f_\alpha\}$ in $C$ that order converges to $f$. By passing to a subnet, we may assume that $\{f_\alpha\}$ is order bounded, and hence it is norm bounded. Since $|f_\alpha-f|\xrightarrow{o}0$, $|f_\alpha-f|\xrightarrow{\sigma(X,X_n^\sim)}0$. By Lemma \ref{orderclosednotsigmaclosed}, we deduce that $f\in C$. Thus, $C$ is an order closed convex set in $X$ which is not $\sigma(X,X_n^\sim)$-closed. This contradicts property $P1$. Thus, either $ X$ or $ X_n^\sim$ is order continuous
\end{proof}

Lemma \ref{orderclosednotsigmaclosed} also gives a characterization of the Krein-Smulian property for \linebreak $\sigma( X, X_n^\sim)$. We say that $\sigma( X, X_n^\sim)$ has the Krein-Smulian property if every convex set $C$ in $ X$ is $\sigma( X, X_n^\sim)$-closed whenever $C\cap k\mathcal{B}$ is $\sigma( X, X_n^\sim)$-closed for every $k\in \mathbb{N}$, where $$\mathcal{B}=\{x\in X: \langle x',|x|\rangle\leq 1\ \text{for all}\ x'\in (X_n^\sim)_+\ \text{with}\ \|x'\|_{X_n^\sim}\leq 1\}.$$ Note that $\mathcal{B}$ is $\sigma(X, X_n^\sim)$-closed and hence, order closed. When $X$ is a Dedekind complete Banach lattice with the weak Fatou property and $X_n^\sim$ separates points of $X$, $\mathcal{B}$ is norm bounded (see \cite{MN}[Theorem 2.4.20]).

\begin{theorem}\label{KSprop}
	Suppose that $ X$ is monotonically complete and $X_n^\sim$ separates points of $X$. Then $\sigma( X, X_n^\sim)$ has the Krein-Smulian property if and only if either $X$ or $X_n^\sim$ is order continuous.
\end{theorem}
\begin{proof}

Suppose that either $X$ or $X_n^\sim$ is order continuous. If $X$ is order continuous, then $\sigma( X, X_n^\sim)$ is just the weak topology and hence, it has the Krein-Smulian property. If $X_n^\sim$ is order continuous, $(X_n^\sim)_n^\sim=( X_n^\sim)^*$. Since $X$ is monotonically complete, $(X_n^\sim)_n^\sim=X$. It follows that $\sigma( X, X_n^\sim)=\sigma(( X_n^\sim)_n^\sim, X_n^\sim)=\sigma(( X_n^\sim)^*, X_n^\sim)$ is the weak-star topology. Hence, it also has the Krein-Smulian property. The reverse implication follows from Lemma \ref{orderclosednotsigmaclosed} and the fact that $\mathcal{B}$ is norm bounded and $\sigma(X, X_n^\sim)$-closed.
\end{proof}

What about the converse of Theorem \ref{sufficient-P1}? It is easy to see that order continuity of $X$ is a sufficient condition for property $P1$.

\begin{proposition}\label{a-orderclosed=sigmaclosed-orderctsnorm}
	A $\sigma$-Dedekind complete Banach lattice $X$ is order continuous if and only if $\overline{C}^o= \overline{C}^{\sigma( X, X_n^\sim)}$ for every convex set $C$ in $ X$. In particular, if $ X$ is order continuous, then $ X$ has property $P1$.
\end{proposition}
\begin{proof}
	Suppose that $\overline{C}^o= \overline{C}^{\sigma( X, X_n^\sim)}$ for every convex set $C$ in $ X$. Then $\overline{Y}^o= \overline{Y}^{\sigma( X, X_n^\sim)}$ for every sublattice $Y$ in $ X$. By \cite[Theorem 2.7]{GL}, $ X$ is order continuous. Conversely, if $ X$ is order continuous, then $ X^*= X^\sim_n$ and hence, $ \sigma( X, X_n^\sim)$ is the weak topology on $ X$. By Mazur's theorem, $ \overline{C}^{\sigma( X, X_n^\sim)}=\overline{C}^{\|\cdot\|_{ X}}$. Since every norm convergent sequence has a subsequence that order converges to the same limit (see, e.g., \cite[Lemma 3.11]{GX1}), we conclude that $\overline{C}^{\sigma( X, X_n^\sim)}\subseteq\overline{C}^o$. Therefore, $\overline{C}^o= \overline{C}^{\sigma( X, X_n^\sim)}$. The second part is an immediate consequence of the first part.
\end{proof}

When $X_n^\sim$ is order continuous, we have the following equivalence condition for property $P1$.

\begin{proposition}\label{orderclosed-normorderclosed-orderctsdual}
	Suppose that $X$ is monotonically complete. Assume that $X_n^\sim$ is order continuous and separates points of $X$. Then $X$ has property $P1$ if and only if every norm bounded order closed convex set in $ X$ is $\sigma( X, X_n^\sim)$-closed.
\end{proposition}
\begin{proof}
The "only if" part is clear. Conversely, suppose that every norm bounded order closed convex set in $ X$ is $\sigma( X, X_n^\sim)$-closed. Let $C$ be an order closed convex set in $X$. Since $\mathcal{B}$ is a norm bounded order closed convex set in $X$, each $C\cap k\mathcal{B}$ is also a norm bounded order closed convex set in $X$. By the hypothesis, each $C\cap k\mathcal{B}$ is $\sigma(X, X_n^\sim)$-closed. Since $\sigma( X, X_n^\sim)$ has the Krein-Smulian property (Theorem \ref{KSprop}), we conclude that $C$ is $\sigma( X, X_n^\sim)$-closed.
\end{proof}

In case that $X$ is an Orlicz space, every norm bounded order closed convex set in $ X$ is $\sigma( X, X_n^\sim)$-closed (\cite[Theorem 3.4]{GLX}). In fact, $\overline{C}^o=\overline{C}^{|\sigma|_s( X, X_n^\sim)}=\overline{C}^{\sigma( X, X_n^\sim)}$ for every norm bounded convex set $C$ in $ X$.  However, this property may fail in a general Banach lattice (see Section 3.2). Proposition \ref{orderclosed-normorderclosed-orderctsdual} motivates us to investigate the following properties.

\begin{definition}
	A Banach lattice $ X$ is said to have property 
	\begin{enumerate}
		\item[$P2$] if every norm bounded order closed convex set in $ X$ is $\sigma( X, X_n^\sim)$-closed. 
		\item[$P3$] if $\overline{C}^o=\overline{C}^{\sigma( X, X_n^\sim)}$ for every norm bounded convex set $C$ in $ X$.
		\item[$P4$] if every order closed norm bounded convex set in $ X$ is $|\sigma|( X, X_n^\sim)$-sequentially closed. 
		\item[$P5$] if $\overline{C}^o=\overline{C}^{|\sigma|_s( X, X_n^\sim)}$ for every norm bounded convex set $C$ in $ X$.
	\end{enumerate}
\end{definition}

The following relations are either immediate or follow from one of Theorem \ref{sufficient-P1}, Proposition \ref{a-orderclosed=sigmaclosed-orderctsnorm} or Proposition \ref{orderclosed-normorderclosed-orderctsdual}.

\begin{figure}[htbp]
\begin{center}
\begin{tikzcd}[arrows=Rightarrow]
{}&{}& X\ \text{or}\  X_n^\sim \text{ is OC}&{}\\
P3 \arrow{r}{} \arrow{d}{}&P2 \arrow[yshift=0.7ex]{r}{\substack{ \text{MC+ ($*$)}\\ +X_n^\sim\ \text{is OC}}}  \arrow{d}{}  & P1 \arrow[l, yshift=-0.7ex] \arrow[shift right, swap]{u}{\substack{\text{DC + wk Fatou+($*$)}}} &  X\ \text{is OC} \arrow{l}\\
P5 \arrow{r}{}&P4 & {} 
\end{tikzcd}
\caption{Basic relations between $P1$, $P2$, $P3$, $P4$ and $P5$. \newline(OC = order continuous, MC = monotonically complete, DC = Dedekind complete, ($*$) = $X_n^\sim$ separates points of $X$)}
\label{fig:basic relations}
\end{center}
\end{figure}
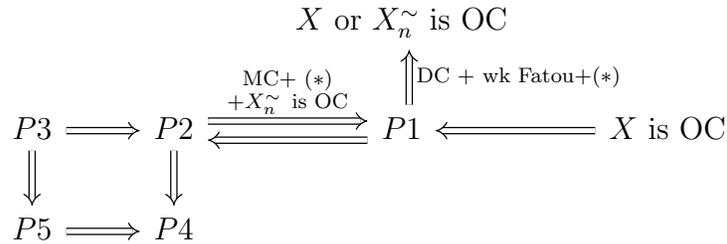

In case that $ X_n^\sim$ is order continuous and contains a strictly positive element, $P2$ is equivalent to $P4$ and $P3$ is equivalent to $P5$. 

\begin{proposition}\label{order-dual-oc-prop}
	Suppose that $ X_n^\sim$ is order continuous and contains a strictly positive element. Then $\overline{C}^{|\sigma|_s( X, X_n^\sim)}=\overline{C}^{\sigma( X, X_n^\sim)}$ for every norm bounded convex set $C$ in $ X$. In particular, $P2$ is equivalent to $P4$ and $P3$ is equivalent to $P5$.
\end{proposition}
\begin{proof} Let $C$ be a norm bounded convex set in $ X$ and $f\in \overline{C}^{\sigma( X, X_n^\sim)}$. According to \cite[Theorem 4.1]{GLX1}, there exists a sequence $\{f_n\}\subseteq C$ such that $|f_n-f|\xrightarrow{uo} 0$. Since $\{f_n\}$ is norm bounded, $|f_n-f|\xrightarrow{\sigma( X, X_{uo}^\sim)}0$. Since $X_n^\sim$ is order continuous, $X_n^\sim=(X_n^\sim)_a=X_{uo}^\sim$. Thus, $|f_n-f|\xrightarrow{\sigma( X, X_n^\sim)}0$ and hence, $f\in \overline{C}^{|\sigma|_s( X, X_n^\sim)}$. This shows that $\overline{C}^{\sigma( X, X_n^\sim)}\subseteq \overline{C}^{|\sigma|_s( X, X_n^\sim)}$. The reverse inclusion is clear. The second part follows directly from the first part.
\end{proof}

Now, we analyze the weakest of properties in Figure \ref{fig:basic relations}, namely $P4$. 

\begin{proposition}\label{cvxnormbddseq-oclosed}
	Let $\{f_n\}$ be a norm bounded disjoint sequence in $ X_+$.  If $\{f_n\}$ is isomorphic to $\ell^1$ basis, then 
	$\overline{\text{co}(\{f_n\})}^o$ is order closed and 
	\[
	\overline{\text{co}(\{f_n\})}^o=\left\{\sum_{i=1}^\infty a_if_i:a_i\geq0, \sum_{i=1}^\infty a_i=1  \right\}.
	\]
\end{proposition}
\begin{proof}
	Let $C:=\left\{\sum_{i=1}^\infty a_if_i:a_i\geq0, \sum_{i=1}^\infty a_i=1  \right\}$. Since
	\[
	\text{co}(\{f_n\})\subseteq C\subseteq \overline{\text{co}(\{f_n\})}^{\|\cdot\|_{ X}}\subseteq \overline{\text{co}(\{f_n\})}^o,
	\]
	we only need to show that $C$ is order closed. Let $h\in \overline{C}^o$. Then there exists a net $\{h_\alpha\}\subseteq C$ such that $h_\alpha\xrightarrow{o}h$. After passing to a subnet, we may assume that $\{h_\alpha\}$ is order bounded, i.e., $h_\alpha\leq g$ for some $g\in X$. Write
	\[
	h_\alpha=\sum_{i=1}^\infty a_{\alpha,i}f_i
	\]
	for some $a_{\alpha,i}\geq0$ with $\sum_{i=1}^\infty a_{\alpha,i}=1$. Clearly, $\{a_{\alpha,i}\}_i\in\ell^1$ for every $\alpha$. For every $i$, let $a_i:=\sup_{\alpha}a_{\alpha,i}$. Clearly, $a_i\in\mathbb{R}_+$ since  $\{a_{\alpha,i}f_i\}_\alpha$ is order bounded (by $g$). Since $\{f_n\}$ is isomorphic to $\ell^1$ basis, there is an $M>0$ such that
	\[
	\frac{1}{M}\sum_{i=1}^k|c_i|\leq \left\|\sum_{i=1}^kc_if_i\right\|_{X}
	\]
	for every $k$ and $\{c_i\}$. It follows that
	\[
	\frac{1}{M}\sum_{i=1}^k|a_i|\leq \left\|\sum_{i=1}^ka_if_i\right\|_{X}\leq \|g\|_{X}
	\]
	for every $k$. This implies that $\{a_i\}_i\in\ell^1$ and hence, $\{\{a_{\alpha,i}\}_i\}_\alpha$ is an order bounded net in $\ell^1$. Note that order intervals in $\ell^1$ are norm compact. Then $\{\{a_{\alpha,i}\}_i\}_\alpha$ is relatively norm compact in $\ell^1$ and hence, it has a subnet $\{\{a_{\alpha_\beta,i}\}_i\}_\beta$ that norm converges to some $\{b_i\}_i\in \ell^1$. Since $\{a_{\alpha,i}\}_\alpha$ is order bounded in $\ell^1$, $\{\{a_{\alpha_{\beta},i}\}_i\}_\beta$ also order converges to $\{b_i\}_i$. It follows that $b_{i}\geq0$ for every $i$ and $\sum_{i=1}^\infty b_{i}=1$. Let
	\[
	\bar{h}=\sum_{i=1}^\infty b_{i}f_i
	\]
	Then $\bar{h}\in C$ and $h_{\alpha_\beta}\xrightarrow{o}\bar{h}$. Since $h_\alpha\xrightarrow{o}h$, we deduce that $h=\bar{h}\in C$. Thus, $C$ is order closed.	
\end{proof}

\begin{corollary}\label{0notinconvexfn}
	Let $\{f_n\}$ be a norm bounded disjoint sequence in $ X_+$.  If $\{f_n\}$ is isomorphic to $\ell^1$ basis, then $0\notin \overline{\text{co}(\{f_n\})}^o$. 
\end{corollary}

\begin{proposition}\label{disjointseq-o-sigma-seq}
	Suppose that $ X$ has property $P4$. Let $\{f_n\}$ be a norm bounded disjoint sequence in $ X_+$.  If $\{f_n\}$ is isomorphic to $\ell^1$ basis, then $\overline{\text{co}(\{f_n\})}^o=\overline{\text{co}(\{f_n\})}^{|\sigma|_s( X, X_n^\sim)}$.
\end{proposition}
\begin{proof}
	By Proposition \ref{cvxnormbddseq-oclosed}, $\overline{\text{co}(\{f_n\})}^o$ is order closed. Hence, it is  $|\sigma|( X, X_n^\sim)$-sequentially closed by $P4$. It follows that
	\[
	\overline{\text{co}(\{f_n\})}^{|\sigma|_s( X, X_n^\sim)}\subseteq \overline{\overline{\text{co}(\{f_n\})}^o}^{|\sigma|_s( X, X_n^\sim)}= \overline{\text{co}(\{f_n\})}^{o}\subseteq \overline{\text{co}(\{f_n\})}^{|\sigma|_s( X, X_n^\sim)}.
	\]
	Therefore, $\overline{\text{co}(\{f_n\})}^o=\overline{\text{co}(\{f_n\})}^{|\sigma|_s( X, X_n^\sim)}$. 
\end{proof}

\begin{lemma}\label{weak-l1}
	Let $\{f_n\}$ be a norm bounded disjoint sequence in $ X_+$. Then $\{f_n\}$ converges weakly to 0 if and only if no subsequence of $\{f_n\}$ is isomorphic to $\ell^1$ basis.
\end{lemma}
\begin{proof}${}$
The "only if" part is clear since $\ell^1$ basis is not weakly null. Conversely, suppose that $\{f_n\}$ does not converge weakly to 0. We will show that $\{f_n\}$ has a subsequence which is isomorphic to $\ell^1$ basis. Since $\{f_n\}$ does not converge weakly to 0, there exist $g\in  X^*$, $\epsilon>0$ and a subsequence $\{f_{n_k}\}$ of $\{f_n\}$ such that $\langle g,f_{n_k}\rangle>\epsilon$ for every $k$. It follows that  
	\[
	\left\|\sum_{k=1}^ma_kf_{n_k}\right\|_{ X}=\left\|\sum_{k=1}^m|a_k|f_{n_k}\right\|_{ X}\geq \frac{1}{\|g\|_{ X^*}}\left\langle g,\sum_{k=1}^m|a_k|f_{n_k}\right\rangle\geq \frac{\epsilon}{\|g\|_{ X^*}}\sum_{k=1}^m|a_k|
	\]
	for every $m$ and $\{a_k\}$. Thus, $\{f_{n_k}\}$ is isomorphic to $\ell^1$ basis. 
\end{proof}

The preceding results suggest the following definition.

\begin{definition}
	A Banach lattice $ X$ is said to have the disjoint order continuity property $(DOCP)$ if for every norm bounded disjoint sequence $\{f_n\}$ in $ X_+$, $f_n\xrightarrow{\sigma( X, X_n^\sim)}0$ implies $\{f_n\}$ converges weakly to 0.
\end{definition}

The reason for the terminology comes from the following simple proposition.

\begin{proposition}\label{order-continuity-property}
	A Banach lattice $X$ is order continuous if and only if $X$ is $\sigma$-Dedekind complete  and for every norm bounded sequence $\{f_n\}$ in $ X_+$, $f_n\xrightarrow{\sigma( X, X_n^\sim)}0$ implies $\{f_n\}$ converges weakly to 0.
\end{proposition}
\begin{proof}
	If $ X$ is order continuous, then $X$ is $\sigma$-Dedekind complete and $\sigma( X, X_n^\sim)$ is the weak topology on $X$. It follows that for every norm bounded sequence $\{f_n\}$ in $ X_+$, $f_n\xrightarrow{\sigma( X, X_n^\sim)}0$ implies $\{f_n\}$ converges weakly to 0. Conversely, suppose that $X$ is $\sigma$-Dedekind complete and for every norm bounded sequence $\{f_n\}$ in $ X_+$, $f_n\xrightarrow{\sigma( X, X_n^\sim)}0$ implies $\{f_n\}$ converges weakly to 0. Let $\{f_n\}$ be an order bounded increasing sequence in $ X_+$ and $f$ be the supremum of $\{f_n\}$. Then $\{f-f_n\}$ is a norm bounded sequence in $ X_+$ and $f-f_n \xrightarrow{\sigma( X, X_n^\sim)}0$. By the hypothesis, $\{f_n\}$ converges weakly to $f$. Using the corollary after \cite[Theorem 5.9, p. 89]{Sc}, we obtain that $\{f_n\}$ norm converges to $f$. Thus, every order bounded increasing sequence in $ X_+$ norm converges. By \cite[Theorem 1.1]{Wn}, we conclude that $ X$ is order continuous.
\end{proof}

In case that $ X_n^\sim$ is order continuous, $DOCP$ implies the order continuity of the norm dual $X^*$ of $X$. In fact, we have the following property.

\begin{proposition}\label{dual-orderdual-B1}
	$X^*$ is order continuous if and only if $ X_n^\sim$ is order continuous and $ X$ has $DOCP$.
\end{proposition}
\begin{proof}
	Suppose that $ X^*$ is order continuous. Clearly, $ X_n^\sim$ is order continuous. Recall that $ X^*$ is order continuous if and only if every norm bounded disjoint sequence in $ X_+$ converges weakly to 0 (\cite[Theorem 3.1]{Wn}). It follows that $ X$ has $DOCP$. Conversely, suppose that $ X_n^\sim$ is order continuous and $ X$ has $DOCP$. Let $\{f_n\}$ be a norm bounded disjoint sequence in $X_+$. By \cite[Corollary 3.6]{GTX}, $f_n\xrightarrow{uo}0$ and hence, $\{f_n\}$ $\sigma(X,X_{uo}^\sim)$-converges to 0. Since $X^\sim_n$ is order continuous, $X_{uo}^\sim=X_n^\sim$. Then $\{f_n\}$ $\sigma(X,X_{n}^\sim)$-converges to 0 and by $DOCP$, $\{f_n\}$ converges weakly to 0. Thus, $ X^*$ is order continuous.
\end{proof}

The next theorem shows that any Banach lattice with property $P4$ has $DOCP$.

\begin{theorem}\label{necessaryP4}
	If $X$ has property $P4$, then $X$ has $DOCP$. In particular, any Banach lattice with property $P1$ has $DOCP$.
\end{theorem}
\begin{proof}
	Let $\{f_n\}$ be a norm bounded disjoint sequence in $ X_+$ such that $f_n\xrightarrow{\sigma( X, X_n^\sim)}0$. Suppose that there is a subsequence $\{f_{n_k}\}$ of $\{f_n\}$ which is isomorphic to $\ell^1$ basis. Since $f_{n_k}\xrightarrow{\sigma( X, X_n^\sim)}0$,  we have that $0\in \overline{\text{co}(\{f_{n_k}\})}^{|\sigma|_s( X, X_n^\sim)}$. On the other hand, by Proposition \ref{disjointseq-o-sigma-seq} and Corollary \ref {0notinconvexfn}, $$0\notin \overline{\text{co}(\{f_{n_k}\})}^o=\overline{\text{co}(\{f_{n_k}\})}^{|\sigma|_s( X, X_n^\sim)}, $$ a contradiction. Thus, no subsequence of $\{f_n\}$ is isomorphic to $\ell^1$ basis and by Lemma \ref{weak-l1}, we conclude that $f_n\xrightarrow{\sigma( X, X^*)}0$.
\end{proof}

Lemma \ref{weak-l1} also gives another condition for $X$ to have $DOCP$.

\begin{lemma}\label{Xa-sigma-weak}
	If $\{f_n\}$ is a sequence in ${X}_a$ which  $\sigma({X},{X}_n^\sim)$-converges to 0, then $\{f_n\}$ converges weakly to 0.
\end{lemma}
\begin{proof}
	Let $g\in X^*$ and $g_a$ be the restriction of $g$ on $X_a$. Then $g_a\in (X_a)^*$. Since $X_a$ is order continuous, $(X_a)^*=(X_a)_n^\sim$. Hence, $g_a$ is a norm bounded order continuous functional on ${X}_a$. Since $X_a$ is an ideal in $X$, \cite[Corollary 1.2]{SZ} implies that $g_a$ has a norm preserving order continuous extension $\tilde{g}$ on $X$. Since $\tilde{g}\in X_n^\sim$,  $\{f_n\}\subseteq X_a$ and $f_n\xrightarrow{\sigma( X, X_n^\sim)}0$, we obtain that $\langle g,f_n\rangle=\langle g_a,f_n\rangle=\langle\tilde{g},f_n\rangle\to0$. Thus, we conclude that $\{f_n\}$ converges weakly to 0.
\end{proof}

\begin{proposition}\label{AM-B1}
	Suppose that $( X\slash X_a)^*$ is order continuous. Then $ X$ has $DOCP$.
\end{proposition}
\begin{proof}
	By Lemma \ref{weak-l1}, it is enough to show that for every norm bounded disjoint sequence $\{f_n\}$ in $ X_+$, $f_n\xrightarrow{\sigma( X, X_n^\sim)}0$ implies that $\{f_n\}$ is not isomorphic to $\ell^1$ basis. Let $\{f_n\}$ be a norm bounded disjoint sequence in $ X_+$ such that $f_n\xrightarrow{\sigma( X, X_n^\sim)}0$. Denote by $q$ the canonical map from $ X$ onto $ X\slash X_a$. Since $( X\slash X_a)^*$ is order continuous and $\{q(f_n)\}$ is a norm bounded disjoint positive sequence in $X\slash X_a$, by \cite[Theorem 3.1]{Wn}, $\{q(f_n)\}$ converges weakly to 0. It follows that
	\[
	0\in \overline{\text{co}(\{q(f_n):n\geq k\})}^{\sigma( X\slash X_a,\left( X\slash X_a\right)^*)}=\overline{\text{co}(\{q(f_n):n\geq k\})}^{\|\cdot\|_{ X\slash X_a}}
	\]
	for every $k$. Choose $k_1\geq 1$ and $g_1\in \text{co}(\{f_n:1\leq n< k_1\})$ such that $\|q(g_1)\|\leq \frac{1}{2^1}$. There exist $k_2> k_1$ and $g_2\in \text{co}(\{f_n:k_1\leq n< k_2\})$ such that $\|q(g_2)\|\leq \frac{1}{2^2}$. Continuing this process, we can find a disjoint sequence $\{g_n\}$ in $\text{co}(\{f_n\})$ such that $g_n\xrightarrow{\sigma( X, X_n^\sim)}0$ and $\{q(g_n)\}$ norm converges to 0. For every $n$, there exists $h_n\in X_a$ such that $0\leq h_n\leq g_n$, $\|g_n-h_n\|_{ X}<\|q(g_n)\|_{ X\slash X_a}+\frac{1}{2^n}$. It follows that $\{g_n-h_n\}$ norm converges to 0 and hence, converges weakly to 0 . Since $\{h_n\}\subseteq  X_a$ and $h_n\xrightarrow{\sigma( X, X_n^\sim)}0$, by Lemma \ref{Xa-sigma-weak}, $\{h_n\}$ converges weakly to 0. Then $\{g_n\}$ also converges weakly to 0 and hence,
	\[
	0\in \overline{\text{co}(\{g_n\})}^{\sigma( X, X^*)}\subseteq \overline{\text{co}(\{f_n\})}^{\sigma( X, X^*)}=\overline{\text{co}(\{f_n\})}^{\|\cdot\|_{ X}}\subseteq \overline{\text{co}(\{f_n\})}^{o}.
	\]
	From Corollary \ref{0notinconvexfn}, we conclude that $\{f_n\}$ is not isomorphic to $\ell^1$ basis. 
\end{proof}

Recall that if $ X$ is an Orlicz space, then $X\slash X_a$ is an AM-space and hence, $( X\slash X_a)^*$ is order continuous. By Proposition \ref{AM-B1}, we obtain that any Orlicz space has $DOCP$. 

A Banach lattice $X$ is said to have the subsequence splitting property if every norm bounded sequence in $X$ has a subsequence that splits into an $X$-equi-integrable sequence and a disjoint sequence (\cite{We}). When $X$ is an Orlicz space and $X_n^\sim$ is order continuous, Delbaen and Owari \cite{DO} showed that $X$ has property $P1$ by using the subsequence splitting property. However, even for Orlicz spaces, analyzing property $P1$ in the situation when $X_n^\sim$ is not order continuous requires a different kind of splitting, which we formalize in the next definition.

\begin{definition}
	A Banach lattice $ X$ is said to have the order subsequence splitting property ($OSSP$) if for every norm bounded sequence $\{f_n\}$ in $ X_+$ which uo-converges to 0, there exists a subsequence $\{f_{n_k}\}$ of $\{f_n\}$ satisfying
	\[
	f_{n_k}=x_{k}+y_{k}+z_{k},
	\]
	where $x_{k},y_{k},z_{k}\geq0$, $x_{k}\in X_a$ for all $k$, $\{y_k\}$ is a disjoint sequence and $\{z_{k}\}$ is order bounded.
\end{definition}

Clearly, any order continuous Banach lattice has $OSSP$. Any Orlicz space also has $OSSP$ (see Section 3.1).

\begin{proposition}\label{cbsp-prop}
	Suppose that $X$ has $OSSP$ and $X_n^\sim$ contains a strictly positive element. Then $DOCP, P4$ and $P5$ all are equivalent.
\end{proposition}
\begin{proof}
	It is enough to show that $DOCP$ implies $P5$. Suppose that $ X$ has $DOCP$. Let $C$ be a norm bounded convex set in $ X$. We will show that $\overline{C}^{o}=\overline{C}^{|\sigma|_s( X, X_n^\sim)}$.
	
	Clearly, $\overline{C}^{o}\subseteq\overline{C}^{|\sigma|_s( X, X_n^\sim)}$. Now, let $f\in \overline{C}^{|\sigma|_s( X, X_n^\sim)}$. Then there exists $\{f_n\}\subseteq C$ such that $|f_n-f|\xrightarrow{\sigma( X, X_n^\sim)}0$. Let $g$ be a strictly positive element in $X_n^\sim$. Then $\langle g,|f_n-f|\rangle\to0$. After passing to a subsequence, we may assume that $\langle g,|f_n-f|\rangle<\frac{1}{2^n}$ for every $n$. Using the same steps as in the proof of \cite[Theorem 4.1]{GLX1}, we obtain that $|f_n-f|\xrightarrow{uo}0$. Hence, by $OSSP$, there exists a subsequence $\{f_{n_k}\}$ of $\{f_n\}$  satisfying 
	\[
	|f_{n_k}-f|=x_{k}+y_{k}+z_{k}
	\]
	where $x_{k},y_{k},z_{k}\geq0$, $x_{k}\in X_a$ for all $k$, $\{y_{k}\}$ is a disjoint sequence and $\{z_{k}\}$ is order bounded. Note that $x_k\xrightarrow{\sigma( X, X_n^\sim)}0$, $y_k\xrightarrow{\sigma( X, X_n^\sim)}0$ and $z_k\xrightarrow{o}0$. Since $\{x_k\}\subseteq  X_a$ and  $x_k\xrightarrow{\sigma( X, X_n^\sim)}0$, $\{x_k\}$ converges weakly to 0 by Lemma \ref{Xa-sigma-weak}. By $DOCP$, $\{y_k\}$ also converges weakly to 0. Hence, $\{x_k+y_k\}$ converges weakly to 0. Since
	\[
	0\in \overline{\text{co}(\{x_k+y_k:k\geq m\})}^{\sigma( X, X^*)}=\overline{\text{co}(\{x_k+y_k:k\geq m\})}^{\|\cdot\|_{ X}}
	\]
	for every $m$, we can find a strictly increasing sequence $\{p_m\}$ and $\{a_{mk}\geq 0:k=p_{m-1}+1,\dotsc,p_m, m\in\mathbb{N}\}$ $(p_0:=0)$ such that $\sum_{k=p_{m-1}+1}^{p_m}a_{mk}=1$ and $\left\|\sum_{k=p_{m-1}+1}^{p_m}a_{mk}(x_k+y_k)\right\|\leq\frac{1}{2^m}$ for every $m$. Note that $\sum_{k=p_{m-1}+1}^{p_m}a_{mk}(x_k+y_k)\xrightarrow{o}0$ as $m\to\infty$. Together with the fact that $\{z_k\}$ order converges to 0, we deduce that
	\[
	\sum_{k=p_{m-1}+1}^{p_m}a_{mk}|f_{n_k}-f|\xrightarrow{o}0
	\]
	as $m\to\infty$ and hence,
	\[
	\sum_{k=p_{m-1}+1}^{p_m}a_{mk}f_{n_k}\xrightarrow{o}f
	\]
	as $m\to\infty$. Since $\sum_{k=p_{m-1}+1}^{p_m}a_{mk}f_{n_k}\in C$ for every $m$, we conclude that $f\in \overline{C}^{o}$. 
\end{proof}

 From Proposition \ref{cbsp-prop} and Proposition \ref{order-dual-oc-prop}, we obtain the following corollary.

\begin{corollary}\label{ocd-P2-P5-DOCP-equiv}
	Suppose that $ X_n^\sim$ is order continuous and contains a strictly positive element. If $ X$ has $OSSP$, then $P2,P3,P4,P5$ and $DOCP$ all are equivalent.
\end{corollary}

Now, we come to the main result of this paper.

\begin{theorem}\label{P1-characterization}
	Suppose that $X$ is a monotonically complete Banach lattice and $X_n^\sim$ contains a strictly positive element. If $X$ has $OSSP$, then the following statements are equivalent:
	\begin{enumerate}[$(1)$]
		\item $ X$ has property $P1$.
		\item $ X$ has $DOCP$ and either $ X$ or $ X^\sim_n$ is order continuous.
		\item $ X$ has $DOCP$ and $\sigma( X, X_n^\sim)$ has the Krein-Smulian property.
		\item either $ X$ or $ X^*$ is order continuous.
	\end{enumerate} 
\end{theorem}
\begin{proof}
Since $X_n^\sim$ contains a strictly positive element, $X_n^\sim$ separates points of $X$. By Theorem \ref{KSprop} and Proposition \ref{dual-orderdual-B1}, $(2),(3)$ and $(4)$ are equivalent. $(1)\implies (2)$ follows from Theorem \ref{sufficient-P1} and Theorem \ref{necessaryP4}. Assume that $(2)$ holds. If $X$ is order continuous, then it has property $P1$ by Proposition \ref{a-orderclosed=sigmaclosed-orderctsnorm}. If $ X_n^\sim$ is order continuous, then $X$ has property $P1$ by Proposition \ref{orderclosed-normorderclosed-orderctsdual} and Corollary \ref{ocd-P2-P5-DOCP-equiv}. Thus $(2)\implies(1)$.
\end{proof}

The results obtained in this section can be summarized into the following diagram.

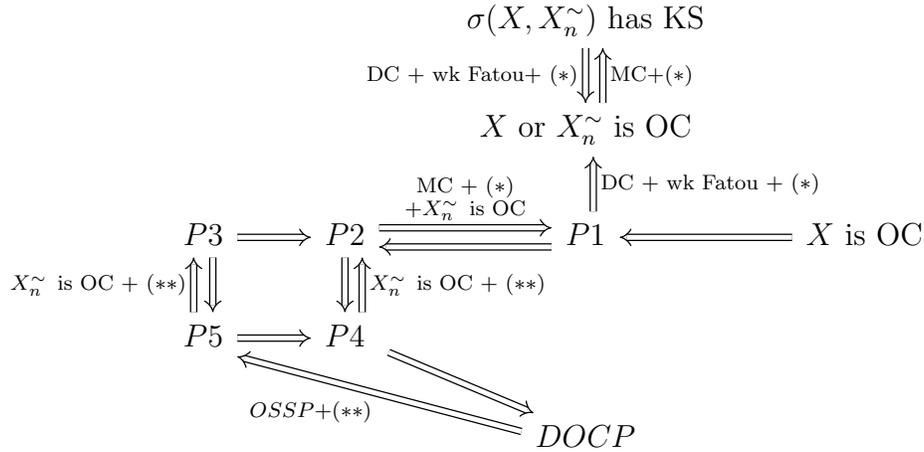
\begin{figure}[htbp]
\begin{center}
\begin{tikzcd}[arrows=Rightarrow]
{}&{}&\sigma( X, X_n^\sim)\ \text{has KS} \arrow[swap]{d}{\text{DC + wk Fatou+ ($*$)}}&{}\\
{}&{}& X\ \text{or}\  X_n^\sim\ \text{is OC} \arrow[xshift=0.7ex, shift right, swap]{u}{\text{MC+($*$)}}&{}\\
P3 \arrow{r}{} \arrow[xshift=0.7ex]{d}{}&P2 \arrow[yshift=0.7ex]{r}{\substack{ \text{MC + ($*$)}\\ + X_n^\sim\ \text{is OC}}}  \arrow{d}{}  & P1 \arrow[l, yshift=-0.7ex] \arrow[shift right, swap]{u}{\text{DC + wk Fatou + ($*$)}} &  X\ \text{is OC} \arrow{l}\\
P5 \arrow{r}{}  \arrow[xshift=-0.7ex]{u}{ X_n^\sim\ \text{is OC + } (**)}&P4 \arrow[xshift=0.7ex]{rd}{} \arrow[xshift=0.7ex, shift right, swap]{u}{ X_n^\sim\ \text{is OC + } (**)}& {}\\
{}&{}&DOCP \arrow[yshift=-0.7ex]{llu}{OSSP + (**)}
\end{tikzcd}
\caption{All relations between $P1$, $P2$, $P3$, $P4$, $P5$ and $DOCP$. (KS = Krein Smulian property, OC = order continuous, MC = monotonically complete, DC = Dedekind complete, ($*$) = $X_n^\sim$ separates points of $X$, ($**$) = $X_n^\sim$ contains a strictly positive element)}
\label{fig:complete relations}
\end{center}
\end{figure}

\section{Some examples}

In this section, we investigate property $P1$ for some known Banach lattices. In the first example, we provide a general class of Banach lattices (including the Orlicz spaces) to which Theorem \ref{P1-characterization} applies. In the second example, we give an example of a space that fails $DOCP$ and in particular, fails both $P1$ and $P2$.

\subsection{Banach lattices with modular conditions}

Let $X$ be a Banach lattice. A functional $\rho:X\to[0,\infty]$ is called a \textit{special modular} on $X$ if it is satisfies the following conditions
\begin{enumerate}[($M$1)]
	\item If $f\in X_+$ and $\rho(f)<\infty$, then $\rho(f_n)\to0$ whenever $f_n\xrightarrow{o}0$ and $0\leq f_n\leq f$.
	\item If $\{f_n\}$ is a sequence in $X$ such that $\sum_{n}\rho(f_n)<\infty$, then a subsequence of $\{f_n\}$ is order bounded.
	\item $\rho(f)<\infty$ for every $f$ in the closed unit ball of $X$.
\end{enumerate}

\begin{example}
Let $(\Omega,\Sigma,\mu)$ be a $\sigma$-finite measure space. Denote by $L^0(\Omega,\Sigma,\mu)$ be the vector lattice of all (equivalence classes with respect to equality a.e. of) real measurable functions on $\Omega$. A function $\varphi:\Omega\times [0,\infty)\to[0,\infty]$ is called a Musielak-Orlicz function if $\varphi(x,\cdot)$ is an Orlicz function for all $x\in\Omega$ and $\varphi(\cdot,t)$ is measurable for all $t\geq0$. For a Musielak-Orlicz function $\varphi$, the functional $\rho_\varphi:L^0(\Omega,\Sigma,\mu)\to[0,\infty]$, given by
	\[
	\rho_\varphi(f)=\int_{\Omega}\varphi(x,|f(x)|)\mathrm{d}\mu,
	\]
	is convex and defines the Musielak-Orlicz space
	\[
	L^\varphi=\{f\in L^0(\Omega,\Sigma,\mu):\rho_\varphi(\lambda f)<\infty\ \text{for some}\ \lambda>0\}
	\]
	with the Luxemburg norm
	\[
	\|f\|_\varphi=\inf\{\lambda>0:\rho_\varphi(f\slash\lambda)\leq1\}.
	\]
	Note that $L^\varphi$ is a Banach function space (i.e., a Banach lattice which is also an ideal of $L^0(\Omega,\Sigma,\mu)$). Furthermore, it is a generalization of Orlicz spaces. One can check that the functional $\rho_\varphi$ is a special modular on $L^\varphi$. See \cite{Mu} and \cite{Wn} for more details on Musielak-Orlicz spaces.
\end{example}

The existence of a special modular $\rho$ on a Dedekind complete Banach lattice $X$ will guarantee $OSSP$. If, in addition, $X$ has the countable sup property and $X_a$ is order dense in $X$, it also will guarantee $DOCP$. A Banach lattice $X$ is said to have the countable sup property if every set $E\subseteq X$ contains a countable subset $E_0$ such that $\sup E=\sup E_0$ whenever $\sup E$ exists in $X$.

\begin{proposition}\label{newnormbounded-closed}
	Let $X$ be a Dedekind complete Banach lattice which supports a special modular. If $\{f_n\}$ is a norm bounded sequence in $ X_+$ which uo-converges to 0, there is a subsequence $\{f_{n_k}\}$ of $\{f_n\}$ with a splitting
	\[
	f_{n_k}=y_k+z_k,
	\]
	where $y_k,z_k\geq0$, $\{y_k\}$ is a disjoint sequence and $\{z_k\}$ is order bounded. In particular, $X$ has OSSP. 
\end{proposition}
\begin{proof}
Let $\{f_n\}$ be a norm bounded sequence in $ X_+$ which uo-converges to 0. Without loss of generality, assume that $\{\|f_n\|_X\}$ is bounded by 1. Denote by $E$ the principal ideal generated by $f:=\sum_n\frac{1}{2^n}f_n$ in $X$. Note that $E$ is a Dedekind complete AM-space with unit $f$ (see \cite[p. 102]{Sc}). Thus there exists a compact Stonian space $K$ (i.e., each open subset of $K$ has open closure) such that $E$ is isomorphic (as a Banach lattice) to $C(K)$, i.e., the space of all real-valued continuous function on $K$. Therefore, we may view every element of $E$, including $f_n$, as a continuous function on $K$. In this case, $f$ can be viewed as the constant function $\rchi_K$ on $K$. Furthermore, for any open-and-closed set $A$ and $h\in C(K)$, $h\rchi_A\in C(K)$.

For each $n\geq 1$, let $D_n:=\overline{\left\{t\in K:f_n(t)>1\right\}}$ and $E_n=\overline{\bigcup_{i=n}^\infty D_i^c}$. Then $D_n$ and $E_n$ are open-and-closed sets in $K$. Let $\rho$ be a special modular on $X$. We claim that for every $n\geq 1$, there exists $m>n$ such that $\rho\left(f_{n}\rchi_{E_m\cap D_n}\right)<\frac{1}{2^n}$. Fix $n\geq 1$. Observe that for every $m>n$, $E_m\cap D_n$ is an open-and-closed set in $K$ and
\[
f_{n}\rchi_{E_m\cap D_n}\leq \|f_{n}\|_\infty \rchi_{E_m}\leq \|f_{n}\|_\infty\sup_{i\geq m}\left(f_i\wedge f\right).
\]
Since $f_m\xrightarrow{uo}0$, $\|f_{n}\|_\infty \sup_{i\geq m}\left(f_i\wedge f\right)\xrightarrow{o}0$ in $X$. It follows that $f_{n}\rchi_{E_m\cap D_n}\xrightarrow{o}0$ in $X$. Note that $0\leq f_n\rchi_{E_m\cap D_n}\leq f_n$ and $\rho(f_n)<\infty$ by $(M3)$. By $(M1)$, we can find $m>n$ such that $\rho\left(f_{n}\rchi_{E_m\cap D_n}\right)<\frac{1}{2^n}$. This proves the claim. 

From the claim, we can find a strictly increasing sequence $\{n_k\}$ such that $$\rho\left(f_{n_k}\rchi_{E_{n_{k+1}}\cap D_{n_{k}}}\right)<\frac{1}{2^k}$$ for every $k$. Set 
$A_k:=D_{n_k}^c$, $B_k:=E_{n_{k+1}}\cap D_{n_{k}}$ and $C_k:=E_{n_{k+1}}^c\cap D_{n_{k}}$. Then for every $k$, $A_k, B_k$ and $C_k$ are disjoint open-and-closed sets in $K$ such that
\begin{enumerate}[(i)]
		\item $A_k\cup B_k\cup C_k=K$,
		\item $f_{n_k}\rchi_{A_k}\leq \rchi_K=f$,
		\item $\rho\left(f_{n_k}\rchi_{B_k}\right)<\frac{1}{2^k}$,
		\item $C_k\subseteq A_{m}\subseteq C^c_{m}$ for every $m>k$.
\end{enumerate}
Define
	\[
	y_{k}:=f_{n_k}\rchi_{C_k},\quad z^{(1)}_{k}:=f_{n_k}\rchi_{A_k} \quad\text{and}\quad z^{(2)}_{k}:=f_{n_k}\rchi_{B_k}.
	\]
Then $y_k,z^{(1)}_k,z^{(2)}_k\in X_+$ and 
	\begin{enumerate}[(i)]
		\item $f_{n_k}=y_{k}+z^{(1)}_{k}+z^{(2)}_{k}$, 
		\item $z^{(1)}_{k}\leq f$,
		\item $\rho(z^{(2)}_{k})<\frac{1}{2^k}$,
		\item $\{y_k\}$ is a disjoint sequence.
	\end{enumerate}
	Clearly, $\{z^{(1)}_{k}\}$ is order bounded by $f$. By $(M2)$, $\{z^{(2)}_{k}\}$ has an order bounded subsequence. Without loss of generality, we may assume that $\{z^{(2)}_{k}\}$ is order bounded.  Hence, $\{z_k\}:=\{z^{(1)}_{k}+z^{(2)}_{k}\}$ is order bounded. Thus, there exists a subsequence $\{f_{n_k}\}$ of $\{f_n\}$ such that
	\[
	f_{n_k}=y_k+z_k
	\]
	where $y_k,z_k\geq0$, $\{y_k\}$ is a disjoint sequence and $\{z_k\}$ is order bounded.
\end{proof}

\begin{proposition}\label{modular-fatou-AMspace}
	Let $X$ be a Banach lattice with the countable sup property. If $X$ admits a special modular and $X_a$ is order dense in $X$, then $X$ has $DOCP$.
\end{proposition}
\begin{proof} Let $\rho$ be a special modular on $X$ and $\{f_n\}$ be a norm bounded disjoint sequence in $X_+$ which $\sigma(X,X_n^\sim)$-converges to 0. Without loss of generality, we assume that $\{f_n\}$ is norm bounded by 1 (hence, $\rho(f_n)<\infty$ by $(M3)$). Since $X_a$ is order dense in $ X$, $\{h:f-h\in  X_a, 0<h<f\}\downarrow0$ for any $f\in  X_+$ (see \cite[ Theorem 3.1]{AB}). By the countable sup property and $(M1)$, we can find a sequence $\{h_n\}$ in $X_a$ such that $0\leq h_n\leq f_n$, $f_n-h_n\in X_a$ and $\rho(h_n)\leq \frac{1}{2^n}$ for every $n$. Since $f_n-h_n\xrightarrow{\sigma(X,X_n^\sim)}0$, $\{f_n-h_n\}$ converges weakly to 0 by Lemma \ref{Xa-sigma-weak}. It remains to show that $\{h_n\}$ converges weakly to 0. 
	
Suppose that $\{h_n\}$ does not converge weakly to 0. By Lemma \ref{weak-l1}, it has a subsequence $\{h_{n_k}\}$ which is isomorphic to $\ell^1$ basis. Since $\{h_{n_k}\}$ is a disjoint sequence, $0\notin \overline{\text{co}(\{h_{n_k}\})}^o$ by Corollary \ref{0notinconvexfn}. On the other hand, since $\sum_k\rho(h_{n_k})<\infty$, $\{h_{n_k}\}$ has an order bounded subsequence by $(M2)$. Note that any disjoint sequence uo-converges to 0 (\cite[Corollary 3.6]{GTX}). It follows that $\{h_{n_k}\}$ has a subsequence which order converges to 0. This implies that $0\in \overline{\text{co}(\{h_{n_k}\})}^o$, a contradiction. Thus, $\{h_n\}$ converges weakly to 0. 
\end{proof}

When $X$ is $\sigma$-Dedekind complete and admits a strictly positive linear functional, it has the countable sup property (see, e.g., the corollary after \cite[Proposition 4.9, p. 78]{Sc}). From Proposition \ref{newnormbounded-closed} and Proposition \ref{modular-fatou-AMspace}, we obtain the following special case of Theorem \ref{P1-characterization}. 

\begin{theorem}
Suppose that $X$ is a monotonically complete Banach lattice and $X_n^\sim$ contains a strictly positive element. If $X$ admits a special modular, then the following statements are equivalent:
	\begin{enumerate}[$(1)$]
		\item $X$ has property $P1$.
		\item Either $X$ or $ X^*$ is order continuous.
	\end{enumerate} 
If, in addition, $X_a$ is order dense in $X$, then they are also equivalent to the following statements: 
\begin{enumerate}[$(1)$]\setcounter{enumi}{2}
	\item Either $ X$ or $ X^\sim_n$ is order continuous.
	\item $\sigma( X, X_n^\sim)$ has the Krein-Smulian property.
\end{enumerate} 
\end{theorem}

Recall that any Banach function space admits a strictly positive order continuous functional (\cite[Proposition 5.19]{GTX}) and in case that $X$ is a Musielak-Orlicz space, $X$ is monotonically complete and $X_a$ is order dense in $X$ (see, e.g., the remark after \cite[Theorem 1.19]{Wn}). Hence, we obtain the following corollaries. Corollary \ref{P1-Musielak-Orlicz} explains why Theorem \ref{P1-characterization} can be seen as a generalization of \cite[Theorem 3.7]{GLX}. 

\begin{corollary}
	Let $X$ be a monotonically complete Banach function space. If $X$ admits a special modular and $X_a$ is order dense in $ X$, then the following statements are equivalent:
	\begin{enumerate}[$(1)$]
		\item $ X$ has property $P1$.
		\item Either $ X$ or $ X^*$ is order continuous.
		\item Either $ X$ or $ X^\sim_n$ is order continuous.
		\item $\sigma( X, X_n^\sim)$ has the Krein-Smulian property.
	\end{enumerate} 
\end{corollary}

\begin{corollary}\label{P1-Musielak-Orlicz}
If $ X$ is a Musielak-Orlicz space, then the following statements are equivalent:
	\begin{enumerate}[$(1)$]
		\item $ X$ has property $P1$.
		\item Either $ X$ or $ X^*$ is order continuous.
		\item Either $ X$ or $ X^\sim_n$ is order continuous.
		\item $\sigma( X, X_n^\sim)$ has the Krein-Smulian property.
	\end{enumerate} 
\end{corollary}

From Figure \ref{fig:complete relations}, one can see that $DOCP$ is the weakest of the properties considered in Section 2. Note that any Musielak-Orlicz space has $DOCP$ (see Proposition \ref{modular-fatou-AMspace}). Hence, it is interesting to see a Banach lattice failing $DOCP$ which we will give in the next example. 

\subsection{Ces\`{a}ro function spaces}

Ces\`{a}ro function spaces $\text{Ces}_p[0,1]$, $1\leq p\leq\infty$, are classes of all Lebesque measurable real functions $f$ on $[0,1]$ such that
\[
\|f\|_{Ces_p[0,1]}=\left[\int_0^1\left(\frac{1}{x}\int_0^x|f(t)|dt\right)^pdx\right]^{1\slash p}<\infty\quad\text{for}\ 1\leq p<\infty
\]
and
\[
\|f\|_{Ces_\infty[0,1]}=\sup_{0<x\leq1}\frac{1}{x}\int_0^x|f(t)|dt<\infty\quad\text{for}\ p=\infty.
\]
The space $\text{Ces}_p[0,1]$ is a Dedekind complete Banach lattice. It is easy to check that for $1\leq p<\infty$, $\text{Ces}_p[0,1]$ is order continuous. For more details on Ces\`{a}ro function spaces see \cite{LZ2} and \cite{AM}.

\begin{theorem}
	Let $ X=\text{Ces}_p[0,1]$. 
	\begin{enumerate}[$(1)$]
		\item If $1\leq p<\infty$, then $ X$ has property $P1$.
		\item If $p=\infty$, then $ X$ does not have $DOCP$. Therefore, $ X$ does not have property $P2$ (and hence $P1$).
	\end{enumerate}
\end{theorem}
\begin{proof}
	If $1\leq p<\infty$, $ X$ is order continuous and by Proposition \ref{a-orderclosed=sigmaclosed-orderctsnorm}, $ X$ has property $P1$. Now, suppose that $p=\infty$. We will show that there exists a norm bounded disjoint sequence $\{f_n\}$ in $ X_+$ such that $f_n\xrightarrow{\sigma( X, X_n^\sim)}0$ but $0\notin\overline{\text{co}(\{f_n\})}^{o}$. Note that this will imply that $ X$ does not have $DOCP$ because otherwise $0\in \overline{\text{co}(\{f_n\})}^{\sigma( X, X^*)}=\overline{\text{co}(\{f_n\})}^{\|\cdot\|_{ X}}\subseteq \overline{\text{co}(\{f_n\})}^{o}$, a contradiction.
	
	Let $b_m=\frac{1}{m}$ and $a_{mn} =b_m-\frac{n(b_m-b_{m+1})}{m+1}$ for $m\in\mathbb{N}$ and $0\leq n\leq m$. Then $b_{m+1}<a_{mn}<\dotsc< a_{m,0}=b_{m}$. Define
	\[
	f_n=\sum_{m=n}^\infty c_{mn}\rchi_{[a_{mn},a_{m,n-1})}
	\]
	where  $c_{mn}=\frac{b_{m+1}-b_{m+2}}{a_{m,n-1}-a_{mn}}$. Observe that $\{f_n\}$ is a disjoint positive sequence and
	\[
	\int_0^xf_n(t)dt=
	\left\{
	\begin{array}{ccl}
	b_{m+2}&\text{if}&b_{m+1}<x<a_{mn},\ m\geq n\\
	c_{mn}(x-a_{mn})+b_{m+2}&\text{if}&a_{mn}\leq x\leq a_{m,n-1},\ m\geq n\\
	b_{m+1}&\text{if}&a_{m,n-1}<x\leq b_{m},\ m\geq n\\
	b_{n+1}&\text{if}&x>b_n
	\end{array}
	\right.
	\]
	for every $n\in\mathbb{N}$ and $x\in (0,1]$. Then
	\[
	\|f_n\|_{Ces_\infty[0,1]}=\sup_{0<x\leq1}\frac{1}{x}\int_0^x|f_n(t)|dt\leq\sup\left\{\sup_{m\geq n}\frac{b_{m+1}}{b_m},1\right\}=1
	\]
	and so, $\{f_n\}$ is norm bounded in $ X$. Recall that the order continuous dual of $ X$ is given as follows:
	\[
	 X_n^\sim=\left(Ces_\infty[0,1]\right)^\sim_n=\tilde{L}^1[0,1]
	\] 
	where $\|f\|_{\tilde{L}^1[0,1]}=\int_0^1\tilde{f}(x)dx$ and $\tilde{f}(x)=\esssup_{t\in[x,1]}|f(t)|$ (\cite[Theorem 4.4]{LZ2}). Therefore, to prove that $f_n\xrightarrow{\sigma( X, X_n^\sim)}0$, it is enough to show that $\int_0^1f_n(t)g(t)dt\to0$ for any decreasing function $g\in L^1[0,1]$. However, this is clear since 
	\[
	\int_0^1f_n(t)g(t)dt\leq\sum_{m=n}^\infty g(b_{m+1})(b_{m+1}-b_{m+2})\leq \int_0^{b_{n+1}}g(t)dt\to0
	\]
	as $n\to\infty$ for every decreasing function $g$ in $L^1[0,1]$.
	
	Now, observe that for every $\{d_n\}$, 
	\begin{eqnarray*}
	\left\|\sum_{n=1}^kd_nf_n\right\|_{Ces_\infty[0,1]}&=&\sup_{0<x\leq1}\frac{1}{x}\int_0^x\sum_{n=1}^k|d_n|f_n(t)dt\\
	&\geq& \frac{1}{b_{k+1}}\int_0^{b_{k+1}}\sum_{n=1}^k|d_n|f_n(t)dt\\
	&=&\frac{b_{k+2}}{b_{k+1}}\sum_{n=1}^k|d_n|\\
	&>&\frac{1}{2}\sum_{n=1}^k|d_n|.
	\end{eqnarray*}
	It follows that $\{f_n\}$ is isomorphic to $\ell^1$ basis. By Corollary \ref{0notinconvexfn}, we deduce that $0\notin \overline{\text{co}(\{f_n\})}^o$.
\end{proof}

\section{Order closedness and  $\sigma( X, X_{uo}^\sim)$-closedness of convex sets in Banach lattices}

Let $ X$ be a Banach lattice. Observe that that $\overline{E}^{o}\subseteq \overline{E}^{\sigma( X, X_{n}^\sim)}\subseteq \overline{E}^{\sigma( X, X_{uo}^\sim)}$ for every set $E$ in $ X$. It follows that every $\sigma( X, X_{uo}^\sim)$-closed convex set is $\sigma( X, X_{n}^\sim)$-closed and order closed.  Then it is natural to ask a problem similar to Problem \ref{mainproblem} for the topology $\sigma( X, X_{uo}^\sim)$:
\begin{problem}\label{sideproblem}
Let $ X$ be a Banach lattice. Is it true that every order closed convex set is  $\sigma( X, X_{uo}^\sim)$-closed?
\end{problem}
In this section, we will give an answer to this problem.

\begin{proposition}\label{uo-P4-orderctsdual}
	Let $ X$ be a Banach lattice. If every norm bounded order closed convex set is $\sigma( X, X_{uo}^\sim)$-closed, then $X^*$ is order continuous.
\end{proposition}
\begin{proof}
	By \cite[Theorem 3.1]{Wn} and Lemma \ref{weak-l1}, it is enough to show that every norm bounded disjoint sequence in $ X_+$ is not isomorphic to $\ell^1$ basis. Suppose not. Then there is a norm bounded disjoint sequence $\{f_n\}$ in $ X_+$ which is isomorphic to $\ell^1$ basis.  By Proposition \ref{cvxnormbddseq-oclosed}, $\overline{\text{co}(\{f_n\})}^o$ is order closed and hence, it is $\sigma( X, X_{uo}^\sim)$-closed by the assumption. Since $\{f_n\}$ is a disjoint sequence, $f_n\xrightarrow{uo}0$ and hence, $f_n\xrightarrow{\sigma ( X, X_{uo}^\sim)}0$. It follows that $0\in  \overline{\text{co}(\{f_n\})}^{o}$. This contradicts Corollary \ref{0notinconvexfn}.
\end{proof}

\begin{theorem}\label{uo-P1}
	Let $ X$ be a Banach lattice. The following statements are equivalent:
	\begin{enumerate}[$(1)$]
		\item Every order closed convex set is $\sigma( X, X_{uo}^\sim)$-closed.
		\item $ X^*$ is order continuous and $ X$ has property $P1$.
		\item $ X^\sim_n$ is order continuous and $ X$ has property $P1$.
	\end{enumerate}
	If $X$ is monotonically complete with $OSSP$ and $X_n^\sim$ contains a strictly positive element, then they are also equivalent to the following:
	\begin{enumerate}[$(1)$]\addtocounter{enumi}{3}
		\item $\overline{C}^o=\overline{C}^{\sigma(X,X_{uo}^\sim)}$ for every norm bounded convex set in $X$.
		\item Every norm bounded order closed convex set is $\sigma( X, X_{uo}^\sim)$-closed.
		\item $X^*$ is order continuous. 
	\end{enumerate}
\end{theorem}
\begin{proof}
Clearly, $(2)\implies  (3)$ and $(4)\implies  (5)$ always hold. $(5)\implies  (6)$ is Proposition \ref{uo-P4-orderctsdual}.\\
	$(1)\implies  (2)$. Since any $\sigma( X, X_{uo}^\sim)$-closed set is  $\sigma( X, X_{n}^\sim)$-closed, $ X$ has property $P1$. By Proposition \ref{uo-P4-orderctsdual}, we also obtain that $ X^*$ is order continuous. \\
	$(3)\implies  (1)$. When $ X_n^\sim$ is order continuous, $ X_{uo}^\sim=X_{n}^\sim$. It follows that $(1)$ and $P1$ are equivalent.\\
	Now, suppose that $X$ is monotonically complete with $OSSP$ and $X_n^\sim$ contains a strictly positive element.\\
	$(3)\implies  (4)$. By Proposition \ref{orderclosed-normorderclosed-orderctsdual} and Corollary \ref{ocd-P2-P5-DOCP-equiv}, $X$ has property $P3$. The conclusion follows from the fact that $X_{uo}^\sim=X_n^\sim$ when $X_n^\sim$ is order continuous.\\ 
	$(6)\implies  (3)$. If $ X^*$ is order continuous, Proposition \ref{dual-orderdual-B1} implies that $ X_n^\sim$ is order continuous and $X$ has $DOCP$. From Theorem \ref{P1-characterization}, we conclude that $X$ has property $P1$.
\end{proof}

From Theorem \ref{uo-P1}, we obtain that Problem \ref{sideproblem} has an affirmative answer if and only if $ X^\sim_n$ is order continuous and $ X$ has property $P1$. In particular, when $ X$ is a Musielak-Orlicz space, Problem \ref{sideproblem} has an affirmative answer if and only if $ X_n^\sim$ is order continuous.


\begin{thebibliography}{100}

\bibitem{AB}
Aliprantis, C., Burkinshaw, O.: Positive Operators. Springer, The Netherlands (2006)

\bibitem{AM}
Astashkin, S.V., Maligranda, L.: Structure of Ces\`{a}ro function spaces -- A survey. Banach Center Publications 102, 13--40 (2014)


\bibitem{BF}
Biagini, S., Frittelli, M.: On the extension of the Namioka-Klee theorem and on the Fatou property for risk measures. In: Optimality and risk - modern trends in mathematical finance, pp. 1--28. Springer, Berlin (2010)


\bibitem{Br}
Brezis, H.: Functional Analysis, Sobolev Spaces and Partial Differential Equations. Springer, Berlin (2010)



\bibitem{De}
Delbaen, F.: Coherent risk measures on general probability
spaces. In: Advances in finance and stochastics, pp. 1--37. Springer, Berlin (2002) 

\bibitem{DO}
Delbaen, F., Owari, K.: On convex functions on the duals of $\Delta_2$-Orlicz spaces. Preprint: {\tt arXiv:1611.06218} 

\bibitem{GL}
Gao, N., Leung, D.: Smallest order closed sublattices and option spanning. Proc. Amer. Math. Soc. 146, 705--716 (2018).

\bibitem{GLMX}
Gao, N., Leung, D., Munari, C., Xanthos, F.: Fatou property, representation, and extension of law-invariant risk measures on general Orlicz spaces. Finance and Stochastics 22, 395--415 (2018)

\bibitem{GLX}
Gao, N., Leung, D., Xanthos, F.: Closedness of convex sets in Orlicz spaces with applications to dual representation of risk measures. Preprint: {\tt arXiv:1610.08806}

\bibitem{GLX1}
Gao, N., Leung, D., Xanthos, F.: Duality for unbounded order convergence and applications. Positivity 22, 711--725 (2018) 

\bibitem{GTX}
 Gao, N., Troitsky, V., Xanthos, F.: Uo-convergence and its applications to Ces\`{a}ro means in Banach lattices. Israel J. Math. 220, 649–-689 (2017)

\bibitem{GX}
 Gao, N., Xanthos, F.: On the C-property and $w^*$-representations of risk
measures. Mathematical Finance 28, 748--754 (2018)

\bibitem{GX1}
 Gao, N., Xanthos, F.: Unbounded order convergence and application to martingales without probability. J. Math. Anal. Appl. 415(2), 931-947 (2014)


\bibitem{LZ2}
Luxemburg, W.A.J., Zaanen, A.C: Some examples of normed K\"{o}the spaces. Math. Annalen 162, 337-350 (1966)

\bibitem{MN}
Meyer-Nieberg, P.: Banach Lattices. Universitext. Springer, Berlin (1991)

\bibitem{Mu}
 Musielak, J.: Orlicz Spaces and Modular Spaces --  Lecture Notes in Mathematics Vol. 1034. Springer, Berlin (1983).


\bibitem{Ow}
 Owari, K.: Maximum Lebesgue extension of monotone convex functions. J. Func. Anal. 266(6), 3572--3611 (2014)


\bibitem{Sc}
 Schaefer, H.H.: Banach Lattices and Positive Operators. Springer, Berlin (1974)

\bibitem{SZ}
 Schaefer, H.H., Zhang, X.D.: Extension properties of order continuous functionals and applications to the theory of Banach lattices. Indag. Math. 5, 107-118 (1994)

\bibitem{We}
Weis, L.: Banach lattices with the subsequence splitting property. Proc. Amer. Math. Soc. 105, 87--96 (1989).

\bibitem{Wn}
 Wnuk, W.: Banach Lattices with Order Continuous Norm. Polish Scientific Publishers, Warszawa (1999)




\end{thebibliography}
\end{document}